\documentclass[a4paper,leqno,11pt]{article}

\usepackage{
amsmath,
amssymb,
amsthm,
verbatim,
fancyhdr,
amsfonts,
}

\usepackage{tikz}
\usetikzlibrary{decorations.pathreplacing}

\newcommand\dx{{\,dx}}
\newcommand\G{\mathcal{G}}
\newcommand\K{\mathcal{K}}
\newcommand\R{\mathbb{R}}
\newcommand\N{\mathbb{N}}

\newtheorem{theorem}{Theorem}[section]
\newtheorem{proposition}[theorem]{Proposition}
\newtheorem{lemma}[theorem]{Lemma}

\theoremstyle{remark}
\newtheorem{remark}[theorem]{Remark}

\theoremstyle{definition}
\newtheorem{definition}[theorem]{Definition}

\tikzstyle{nodo}=[circle,draw,fill,inner sep=0pt,minimum size=.1cm]
\tikzstyle{infinito}=[circle,inner sep=0pt,minimum size=0mm]

\date{}

\title{Bound states of the NLS equation on metric graphs\\with localized nonlinearities}

\author{Enrico Serra, Lorenzo Tentarelli
\\ \ \\
{\small  Dipartimento di Scienze Matematiche ``G.L. Lagrange'', Politecnico di Torino } \\
{\small Corso Duca degli Abruzzi, 24, 10129 Torino, Italy} \\
{\small \texttt{enrico.serra@polito.it}, \texttt{lorenzo.tentarelli@polito.it}}}

\begin{document}

\maketitle

\begin{abstract}
We investigate the existence of multiple bound states of prescribed mass for the nonlinear Schr\"odinger equation on a noncompact metric graph. The main feature is that the nonlinearity is localized only in a compact part of the graph.
Our main result states that for every integer $k$, the equation possesses at least $k$ solutions of prescribed mass, provided that the mass is large enough. These solutions arise as constrained critical points of the NLS energy functional. Estimates for the energy of the solutions are also established.
\end{abstract}

\smallskip


\section{Introduction}

In this paper we discuss the existence of multiple bound states of {\em prescribed mass} for the NLS equation on metric graphs with localized nonlinearities.

Here we limit ourselves to a rather informal description of the results: the precise setting of the problem with all the required definitions will be given in detail in Section \ref{set_res}.

We consider a \emph{noncompact connected metric graph}, that is, a connected graph  $\G=(\mathrm{V},\mathrm{E})$ where each edge is associated with a (possibly unbounded) interval of the real line. 
The subgraph $\K$ of $\G$ consisting of all the bounded edges  is called  the {\em compact core} of $\G$ (see Definition \ref{cc}). In Figure \ref{figuno} is depicted a typical metric graph.

Given a number $\mu> 0$, a bound state of {\em mass} $\mu$ for the NLS equation on $\G$, with nonlinearity localized in the compact core,  is a function $u\in H^1(\G)$ satisfying
\begin{equation}
\label{bsmass}
\|u\|_{L^2(\G)}^2 = \mu
\end{equation}
and that solves, for a suitable $\lambda \in \R$, the stationary NLS equation
\begin{equation}
\label{bseq}
u''+\kappa(x)|u|^{p-2}\,u = \lambda u
\end{equation}
on every edge of $\G$ with homogeneous Kirchhoff conditions at the vertices of the compact core $\K$. The function $\kappa$ is the characteristic function of the compact core (this explains the term {\em localized} referred to the nonlinearity). 
 
Many different boundary conditions can be imposed at the vertices of the compact core, giving rise to different models (see e.g. \cite{berkolaiko}). In this paper we choose the Kirchhoff conditions, since they are the most commonly used in the applications. See Definition \ref{defbs} for the precise statement.

Bound states satisfy therefore a  {\em double regime}: linear in the unbounded edges and nonlinear in the compact core. 

In this paper we only treat nonlinearities with {\em subcritical} growth, namely, we always assume that $2<p<6$.

\begin{figure}[ht]
\begin{center}
\begin{tikzpicture}[xscale= 0.7,yscale=0.7]
\node at (-.5,2) [nodo] (02) {};
\node at (2,2) [nodo] (22) {};
\node at (2,4) [nodo] (24) {};
\node at (3.6,1.6) [nodo] (42) {};
\node at (3,3) [nodo] (33) {};
\node at (5,2) [nodo] (52) {};
\node at (3,3) [nodo] (32) {};
\node at (4,3) [nodo] (43) {};
\node at (-1,4) [nodo] (04) {};
\node at (2,4) [nodo] (24) {};
\node at (.5,1) [nodo] (11) {};
\node at (2,0) [nodo] (20) {};
\node at (4,0) [nodo] (40) {};
\node at (-4,2) [minimum size=0pt] (meno) {};
\node at (-4,4) [minimum size=0pt] (menoalt) {};
\node at (9,2) [minimum size=0pt] (piu) {};
\node at (-4.1,2) [infinito]  (infmeno) {$\infty$};
\node at (-4.1,4) [infinito]  (infmenoalt) {$\infty$};
\node at (9.1,2) [infinito]  (infpiu) {$\infty$};
\draw[-] (02)--(04);
\draw[-] (04)--(24);
\draw[-] (04)--(22);
\draw[-] (24)--(22);
\draw[-] (02)--(11);
\draw[-] (11)--(22);
\draw[-] (11)--(20);
\draw[-] (20)--(22);
\draw[-] (22)--(33);
\draw[-] (24)--(33);
\draw[-] (24)--(43);
\draw[-] (33)--(43);
\draw[-] (43)--(52);
\draw[-] (33)--(42);
\draw[-] (20)--(42);
\draw[-] (20)--(40);
\draw[-] (40)--(42);
\draw[-] (42)--(52);
\draw[-] (40)--(52);
\draw[-] (02)--(meno);
\draw[-] (52)--(piu);
\draw[-] (04)--(menoalt);
\end{tikzpicture}
\end{center}
\caption{\footnotesize{{a metric graph with 3 half--lines and 19 bounded edges. The bounded edges form the {\em compact core} of the graph, where the nonlinearity is localized.}
 }}
\label{figuno}
\end{figure}
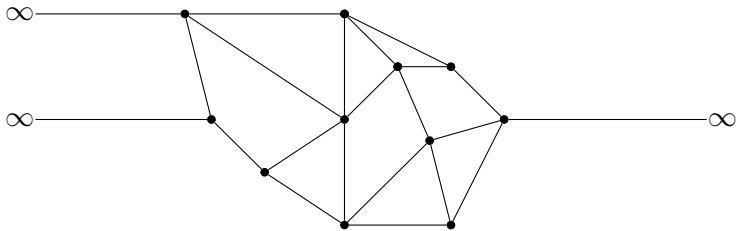

The existence of solutions with prescribed mass to stationary NLS type equations is a classical subject, both in dimension one and in higher dimensions, and the literature is huge (see e.g \cite{cazenave,caz-lio,suldue} and references therein).
From a variational point of view, bound states of prescribed mass (i.e. the solutions of \eqref{bsmass}--\eqref{bseq}) can be found as critical points of the energy functional
\begin{equation}
\label{functe}
E(u) = \frac12\|u'\|_{L^2(\G)}^2 - \frac1p\|u\|_{L^p(\K)}^p = \frac12\int_\G |u'|^2\dx - \frac1p\int_\K |u|^p\dx,
\end{equation}
constrained on the manifold
\begin{equation}
\label{vinc}
M= \{u\in H^1(\G) : \|u\|_{L^2(\G)}^2 = \mu\}.
\end{equation}

The principal obstacle that one has to face when looking for critical points of functionals like $E$ is given by the lack of compactness caused by the unboundedness of the domain. As is well known, this may result in the failure of convergence of minimizing or Palais--Smale sequences for $E$. 
Typical devices used to gain compactness, such as restricting to functions enjoying some kind of symmetry, do not work on general non symmetric  graphs like ours. 

A further difficulty, typical of problems on graphs, is the impossibility of obtaining solutions of prescribed mass by {\em scaling}. For instance, 
when the spatial domain is $\R$ or $\R^n$, a possible approach to find (multiple) solutions consists in setting up  suitable min--max procedures and then  obtaining solutions of prescribed mass by scaling critical points  to adjust their $L^2$ norm. In these cases one may even expect infinitely many solutions of prescribed mass (see e.g. \cite{bariola,bl2}).
In order to carry out this program,
it is of course necessary that the domain be invariant under homotheties. In our case however the presence of bounded edges of fixed length prevents this possibility.
A striking example of this phenomenon can be found in \cite{T}, where it is shown that when $p\in [4,6)$, the functional \eqref{functe} does not even  admit   a global {\em  minimum} on $M$ for small values of the mass $\mu$.

The scheme of the procedure we use to find multiple bound states can be roughly summarized as follows. We consider the restriction $E_M$ of the functional $E$ to the manifold $M$ defined in \eqref{vinc}. First we prove that $E_M$ satisfies the Palais--Smale condition at a level $c$ if and only if $c<0$ (Section \ref{palais}). Then we construct, under suitable assumptions on $\mu$, multiple min-max classes working at negative levels (Section \ref{proofs}). The construction of these classes takes advantage of the evenness of $E_M$ and is based on the Krasnosel'skii genus. More precisely, we show that for every $k\in \N$,   some {\em negative} sublevels of $E_M$ 
have genus at least $k$ provided the mass  $\mu$ is large enough. This yields the existence of at least $k$ (pairs of) critical points for $E_M$, by standard results in Critical Point Theory.

A fundamental role in the construction of the min--max classes is played by the properties of the {\em solitons} $\varphi_\mu$, namely, the minimizers of the functional
\begin{equation}
\label{fusol}
\mathcal{E}(u) = \frac12\int_{\R}|u'|^2\dx - \frac1p\int_{\R}|u|^p\dx, \qquad p\in(2,6),
\end{equation}
over the set of  $u\in H^1(\R)$ such that $\|u\|_{L^2(\R)}^2 = \mu$ (Section \ref{solitons}). These will be used, after being appropriately scaled, cut off and placed in the compact core of $\G$, to show that some negative sublevel sets of $E_M$ have the required genus.

\begin{remark} The topological argument that we use (detection of sublevels of high genus) works also when the nonlinearity is present on the half--lines. In this case however it can be proved that  the Palais--Smale condition fails at infinitely many negative levels. Thus any variational argument for the everywhere nonlinear problem must avoid the  ``bad levels''. This seems to be out of reach at the moment.   
\end{remark}

Our main result is the following.

\begin{theorem}
\label{main}
Let $\G$ be a noncompact connected metric graph with nonempty compact core, and let $p\in(2,6)$. For every $k\in\N$, there exists ${\mu_k}>0$ such that for all $\mu \ge {\mu_k}$ there exist at least $k$ distinct pairs $(\pm u_j)$ of bound states of mass $\mu$. Moreover, for every $j=1, \dots, k$, 
\begin{equation}
\label{bslevel}
E_M(u_j) \le \frac1{j^{ 2\beta}}\,\mathcal{E}(\varphi_\mu) + \sigma_k(\mu) < 0, \qquad \beta = \frac{p-2}{6-p},
\end{equation}
where $\sigma_k(\mu) \to 0$ (exponentially fast) as $\mu \to \infty$. Finally, for each $j$, the Lagrange multiplier $\lambda_j$ relative to $u_j$ is positive.
\end{theorem}

This result shows, in particular, that the localization of  the nonlinearity in the compact core  is essential. Consider for instance the simple case where $\G$ is an interval with two half--lines attached at its endpoints. In the degenerate case where the interval is reduced to a point equation
\eqref{bseq} becomes linear and there are of course no bound states, for any $\mu>0$. Also in the other degenerate case, where the interval is the whole real line, the functional
$E$ reduces to $\mathcal{E}$ of \eqref{fusol} and this has infinitely many ground states (the solitons), but no bound states at higher levels. It is thus the presence of the nonlinearity on a compact portion of the graph  only that generates the bound states.

\begin{remark}
We observe that since by the properties of solitons (Section \ref{solitons}),  $\frac1{j^{ 2\beta}}\,\mathcal{E}(\varphi_\mu) = j\mathcal{E}(\varphi_{\mu/j})$, the level estimate \eqref{bslevel} means, asymptotically in $\mu$, that the ``j-th" bound state $u_j$  behaves, energetically, not worse than the superposition 
of $j$ solitons of mass $\mu/j$, far apart from each other.
\end{remark}

\begin{remark}
As we will see in Section \ref{proofs}, the solution corresponding to $j=1$ is a global minimum of $E_M$, namely a {\em ground state}. This completes and clarifies some results of \cite{T}: in some cases ground states exist if and only if their mass is large enough and their level, when $\mu\to \infty$, is asymptotically lower than that of a soliton of the same mass.
\end{remark}

\begin{remark} 
\label{single}
We point out that we have stated our main result, Theorem \ref{main}, under the assumption that the nonlinearity be localized in the {\em whole} compact core of the graph. We have done so because this is the most meaningful case in the applications. However, the technique of the proofs allows us to obtain the same result by assuming that the nonlinearity is localized in a {\em subset} of the compact core, even in a single edge  (the case treated numerically in \cite{der}). No change is necessary to carry out the proof in this more general case.
\end{remark}

\begin{remark}
Some of the solutions we find might vanish identically on one or more half--lines, which is compatible with our definition of bound states. This type of solutions have been first found in \cite{CFN} on the tadpole graph for the everywhere nonlinear problem. Ground states, however, cannot have this property, because being of constant sign they would violate the Kirchoff condition at the vertices of the half--lines where they vanish (see \cite{T}).
\end{remark}

Driven by the applications, the study of dynamics on quantum graphs has gained popularity in recent years since graphs emulate accurately complex mesoscopic and optical networks and allow one to reproduce properties of quantum chaotic systems (see \cite{gnusmi,kottos}).

In particular, \cite{der} highlights some specific motivations for studying the existence of bound states for the NLS equation on metric graphs with nonlinearity localized on the compact core and \cite{noja} proposes the analysis of the problem from a rigorous point of view. 

The investigation of these topics  seems to play an important role in the analysis of the effects of nonlinearities on the transmission through a complex network of one--dimensional leads (such as \emph{optical fibers}) and on the properties of the \emph{Bose--Einstein condensates} in non--regular traps.

For an overview on the Schr\"odinger equation with localized or concentrated  nonlinearities in standard domains we refer the reader to \cite{adft,at,Cacciapuoti,josinio,nier}.

Finally, for the sake of completeness, we also  recall that a general investigation of the existence of {\em ground} states in the ``everywhere nonlinear'' problem (namely, the problem with the nonlinearity located on the whole graph $\G$) is developed in  \cite{acfn,ast1,ast2}.

\section{Functional setting}\label{set_res}

We start by recalling briefly some basic facts and defintions concerning metric graphs (for a modern account on the topic we refer the reader to \cite{Kuchment,berkolaiko,ast1} and references therein).

\medskip
In this paper, by \emph{connected metric graph} $\G=(\mathrm{V},\mathrm{E})$, we mean a connected {\em multigraph} (that is, multiple edges between the same vertices and self--loops are allowed), where each edge is a (possibly half--infinite) segment of line. The edges are joined at their endpoints --- the vertices of $\G$ --- according to the topology of the graph. Each edge $e$ is represented by either a closed bounded interval $I_e =[0,\ell_e]$, $\ell_e>0$, or a closed half--line $I_e = [0,+\infty)$, and a coordinate  $x_e$ is chosen in the corresponding interval. When the interval is bounded, the orientation of $x_e$ is arbitrary while, in case of an unbounded edge, the half--line always starts at $x_e=0$. In this way the graph $\G$ turns into a locally compact metric space, the metric being given by the shortest distance measured along the edges. In the sequel, for the sake of simplicity, we identify an edge $e$ with its corresponding interval $I_e$, since this does not give rise to misunderstandings. 

A metric graph is \emph{compact} if and only if it does not contain any half--line.  An important role in this paper is played by the following notion, introduced in \cite{ast2}.

\begin{definition}
\label{cc}
If $\G$ is a metric graph, we define its \emph{compact core} $\K$ as the metric subgraph of $\G$ consisting of all the bounded edges of $\G$.
\end{definition}

Obviously the compact core $\K$ of any graph $\G$ is compact and, when $\G$ is connected (as we always assume throughout), $\K$ is also connected.

\medskip
A function $u:\G\to\R$ can be regarded as a family of functions $(u_e)_{e\in\mathrm{E}}$, where $u_e:I_e\to\R$ is the restriction of $u$ to the edge (represented by) $I_e$. The usual $L^p$ spaces can be defined over $\G$ in the natural way, with norm
\[
\|u\|_{L^p(\G)}^p := \sum_{e\in\mathrm{E}} \| u_e\|_{L^p(I_e)}^p,
\]
while $H^1(\G)$ is the set of continuous $u:\G\to \R$ such that $u_e\in  H^1(I_e)$ for every $e\in\mathrm{E}$ and
\[
 \|u\|_{H^1(\G)}^2 := \|u'\|_{L^2(\G)}^2 + \|u\|_{L^2(\G)}^2.
\]
We notice that continuity at a vertex $\mathrm{v}$ is to be interpreted as a no--jump condition among different components $u_e$
meeting at $\mathrm{v}$. Further details can be found in \cite {ast1}. From now on we denote the space $H^1(\G)$ simply by $H$ and its norm by $\|\,\cdot\,\|$. Also,  $H'$ denotes the dual of $H$.

\medskip
As we outlined in the Introduction, the object of this paper is to prove existence of multiple bound states of prescribed mass for the NLS equation on
a noncompact metric graph $\G$, where the nonlinearity is localized {\em only} in the compact core $\K$. 
Of course we always assume that  $\K$ is not empty.

The precise definition of bound state of prescribed mass is the following.

\begin{definition}
\label{defbs}
Let $\G = (\mathrm{V},\mathrm{E})$ be a connected metric graph with nonempty compact core $\K$ and let $\kappa:\G\to \{0,1\}$ be the characteristic function of the compact core. 
We say that a function $u\in H$ is a \emph{bound state of mass $\mu$ for the NLS equation on $\G$ with nonlinearity localized on $\K$  and Kirchhoff conditions} if:
\begin{itemize}
\item[(i)] $\|u\|_{L^2(\G)}^2 = \mu$,
\item[(ii)] there exists $\lambda\in\R$ such that for every $e\in\mathrm{E}$
\[
u_e'' + \kappa(x)|u_e|^{p-2}\,u_e = \lambda u_e,
\]
\item[(iii)] for every vertex $\mathrm{v}$ in the compact core $\K$,
\[
\sum_{e\succ\mathrm{v}}\frac{du_e}{dx_e}(\mathrm{v}) = 0.
\]
\end{itemize}
\end{definition}


The symbol $e\succ\mathrm{v}$ means that the sum is extended to all edges $e$ incident at ${\mathrm v}$, while  $\frac{du_e}{dx_e}(\mathrm{v})$ is a shorthand notation for $u_e'(0)$ or $-u_e'(\ell_e)$, according to the fact that $x_e$ is equal to $0$ or $\ell_e$ at $\mathrm{v}$. 

For simplicity, we refer to a function $u$ satisfying Definition \ref{defbs} simply as a \emph{bound state of mass} $\mu$.

\medskip
Throughout this paper we always suppose that  $p\in(2,6)$, the so-called $L^2$--subcritical case.

\medskip
The problem of the search of bound states of mass $\mu$ has a variational structure that we now describe and that we will exploit to prove a multiplicity result. 

Let $E:H \to \R$ be the functional
\[
E(u) = \frac12\|u'\|_{L^2(\G)}^2 - \frac1p\|u\|_{L^p(\K)}^p = \frac12\int_\G |u'|^2\dx - \frac1p\int_\K |u|^p\dx.
\]
Plainly, $E\in C^1(H;\R)$ and
\[
E'(u)v = \int_\G u'v'\dx - \int_\K |u|^{p-2}\,uv\dx \qquad \forall v \in H.
\]
For $\mu >0$, define the manifold
\begin{equation}
\label{manif}
M= \{u\in H : \|u\|_{L^2(\G)}^2 = \mu\}.
\end{equation}
This is a submanifold of $H$ of codimension $1$ with the structure of  Finsler manifold (see e.g. \cite{bl2,Deimling}) and its tangent space at a given point $u$ is the closed subspace of $H$  defined by
\[
T_uM = \left\{v\in H:\int_{\G}uv\dx=0\right\}.
\]
We denote its dual space $(T_uM)'$  by $T_u'M$; the norms in $T_uM$ and in $T_u'M$ are those induced by the norms in $H$ and $H'$, respectively.

Finally, let $E_M: M\to \R$ denote the restriction of $E$ to $M$; of course $E_M$ is $C^1$ on $M$ and
\[
E_M'(u) v = E'(u)v\qquad \forall v \in T_uM.
\]
The following characterization follows immediately from the Lagrange multiplier theorem.

\begin{proposition}
A function $u\in H$ is a bound state of mass $\mu$ if and only if $u\in M$ and $E_M'(u) = 0$.
\end{proposition}

We are thus led to the search of critical points of the functional $E_M$. To simplify some computations in the next sections, we set 
\begin{equation}
\label{lam}
\lambda = \lambda(u) := -\frac1\mu E'(u)u = \frac1\mu \left(\int_\K |u|^p\dx - \int_\G |u'|^2\dx\right)
\end{equation}
and define, for every $u\in M$, a linear functional $J(u) : H\to \R$ as 
\[
J(u)v = \int_\G u'v'\dx - \int_\K |u|^{p-2}\,uv \dx +  \lambda \int_\G uv\dx.
\]
Note that
\begin{equation}
\label{juv}
J(u)v = E_M'(u)v\qquad \forall v \in T_uM
\end{equation}
and, by definition of $\lambda$, 
\begin{equation}
\label{juu}
J(u)u = 0\qquad \forall u \in H.
\end{equation}
It is common to identify $E_M'(u)$ with $J(u)$. However, as pointed out in Remark 8.2 of \cite{bl2}, the identification is not always correct, away from critical points. The precise relation between $E_M'(u)$ and $J(u)$ that we are going to use is given in the following statement, discussed abstractly in \cite{bl2}. We give a proof for completeness.

\begin{lemma}
Let $u\in M$. Then
\begin{equation}
\label{norms}
\|E_M'(u)\|_{T_u'M} \le \|J(u)\|_{H'} \le \left(1+\mu^{-1/2}\|u\| \right)\|E_M'(u)\|_{T_u'M}.
\end{equation}
\end{lemma}

\begin{proof}
By \eqref{juv},
\[
\|E_M'(u)\|_{T_u'M} \le \|J(u)\|_{H'},
\]
proving the first inequality in \eqref{norms}. Let $\pi_u : H \to T_uM$ be the restriction to $H$ of the $L^2$ orthogonal projection onto $T_uM$. Then, every $v\in H$
can be written as $v =\left(\mu^{-1}\int_\G uv\dx \right) u + \pi_u v$ and 
\begin{equation}
\label{ineq}
\|\pi_u v\| \le  \left(1+ \mu^{-1/2}\|u\| \right)\|v\|.
\end{equation}
By \eqref{juv} and \eqref{juu}, for every $v\in H$,
\[
J(u)v = J(u)\left(\left(\mu^{-1}\int_\G uv\dx\right) u + \pi_u v\right) = J(u)\pi_u v = E_M'(u)\pi_u v.
\] 
Then
\[
|J(u)v| \le \|E_M'(u)\|_{T_u'M}\|\pi_u v\| \qquad \forall v \in H.
\]
In view of \eqref{ineq}, this proves also the second inequality in \eqref{norms}.
\end{proof}

\begin{remark}
\label{bdps}
From \eqref{norms} we see that $u\in M$ is a critical point for $E_M$ if and only if $J(u)=0$. Also, the preceding lemma allows us to say that if $(u_n)\subset M$ is a \emph{bounded} sequence such that $\|E_M'(u_n)\|_{T_{u_n}'M} \to 0$ , then $\|J(u_n)\|_{H'} \to 0$, and vice--versa. Note that if the boundedness assumption is dropped, this need no longer be true.
\end{remark}

\section{Some properties of solitons}\label{solitons}

Let $\mu>0$ and $p\in (2,6)$. The \emph{solitons of mass} $\mu$ are the minimizers of the functional
\[
\mathcal{E}(u) = \frac12\int_{\R}|u'|^2\dx - \frac1p\int_{\R}|u|^p\dx
\]
on the set $\{u\in H^1(\R) : \|u\|_{L^2(\R)}^2 = \mu\}$. It is well known (\cite{cazenave, zakharov}) that, up to a change of sign,  these are the translates of a unique  function $\varphi_\mu$ (the dependence on $p$ is understood) which is positive, even and radially decreasing. Precisely, and in order to enlighten the dependence on $\mu$,
\begin{equation}
\label{soliton}
\varphi_\mu(x) = \mu^{\alpha}\varphi_1(\mu^{\beta}x), \qquad  \alpha = \frac{2}{6-p},\quad\beta = \frac{p-2}{6-p},
\end{equation}
where
\begin{equation}
\label{form}
\varphi_1(x) := C_p\mathrm{sech}^{\alpha/\beta}(c_p x), \qquad\mbox{with}\quad C_p, c_p>0.
\end{equation}
A direct computation shows that
\begin{equation}
\label{levsol}
 \mathcal{E}(\varphi_1) < 0 \qquad\mbox{and}\qquad \mathcal{E}(\varphi_\mu) = \mu^{2\beta+1}\mathcal{E}(\varphi_1).
\end{equation}
The function $\varphi_\mu$ is the unique (positive and radially decreasing) solution of the stationary NLS equation
\[
 u''+|u|^{p-2}\,u=\lambda u \qquad\text{ on } \R
\]
with
\begin{equation}
\label{molt}
\lambda = \frac1\mu \left(\|u\|_{L^p(\R)}^p - \|u'\|_{L^2(\R)}^2 \right) > 0
\end{equation}
in the class of functions $u\in H^1(\R)$ with $\|u\|_{L^2(\R)}^2=\mu$. The conservation of energy for $\varphi_\mu$ reads
\begin{equation}
\label{energy}
\frac12|\varphi_\mu'(x)|^2 + \frac1p|\varphi_\mu(x)|^p = \frac\lambda2|\varphi_\mu(x)|^2 \qquad\forall x\in\R.
\end{equation}

Let $\mathcal{L}_\mu(x)$ be the Lagrangian density of $\mathcal{E}(\varphi_\mu)$, namely
\[
\mathcal{L}_\mu(x) = \frac12|\varphi_\mu'(x)|^2 - \frac1p |\varphi_\mu(x)|^p.
\]
The following simple property will be used below.

\begin{lemma}
\label{pos}
For every $\mu>0$ there exists a unique point $x_{\mu}>0$ such that
\[
\mathcal{L}_\mu(x)\ge 0 \iff  |x|\ge x_\mu.
\]
In particular, $x_\mu = \mu^{-\beta}x_1$, so that $x_\mu$ is a decreasing function of $\mu$ and tends to zero as $\mu\to \infty$.
\end{lemma}

\begin{proof}
By \eqref{soliton}, $\mathcal{L}_\mu(x) = \mu^{\alpha p}\mathcal{L}_1(\mu^\beta x)$ (note that $2\alpha+2\beta= \alpha p$). The inequality
\[
\mathcal{L}_1(t) \ge 0,
\]
in view of \eqref{energy}, is equivalent to
\[
\frac2p|\varphi_1(t)|^p \le \frac\lambda2|\varphi_1(t)|^2,
\]
namely,
\[
|\varphi_1(t)|^{p-2}\le \frac{p\lambda}{4}.
\]
Notice that here $\lambda$ no longer depends on $\mu$, being given by \eqref{molt} with $u=\varphi_1$. As $\mathcal{L}_1(0) < 0$, we see that $|\varphi_1(0)|^{p-2} > p \lambda/4$, so that, since $\varphi_1$ is radially decreasing, there is a unique point $x_1$ such that
\[
|\varphi_1(t)|^{p-2} \le \frac{p \lambda}{4} \iff |t| \ge x_1.
\]
Setting $x_\mu = \mu^{-\beta}x_1$, the proof is complete.
\end{proof}

The following proposition is crucial for our purposes. Roughly, it says that when the mass is large, the energy of a soliton can be approximated 
using functions with small support.

\begin{proposition}
\label{approx}
For every $\ell>0$ there exists $\mu_\ell>0$ such that for all $\mu \ge \mu_\ell$, there exists $\psi=\psi_\mu\in H^1(\R)$ satisfying
\begin{itemize}
\item[(i)] $\|\psi\|_{L^2(\R)}^2 = \mu$,
\item[(ii)] ${\rm supp }\,\psi \subset [0, \ell]$,
\item[(iii)] $\mathcal{E}(\psi) \le \mathcal{E}(\varphi_\mu) + r_\ell(\mu) <0$,
\end{itemize}
where $r_\ell(\mu) \to 0$ (exponentially fast) as $\mu \to \infty$.
\end{proposition}

\begin{proof}
Given $\ell >0$, let $\tau = \ell/2$. We take, using Lemma \ref{pos}, $\mu_1$ so large that for every $\mu\ge \mu_1$, the point $x_\mu$ of Lemma \ref{pos}
satisfies $x_\mu\le \tau$. In this way,
\begin{equation}
\label{lapos}
\mathcal{L}_\mu (x) \ge 0 \qquad \text{for every $|x|\ge \tau$ and every $\mu \ge \mu_1$}.
\end{equation}
Then, for each $\mu \ge \mu_1$, we define the function
\[
v_\mu(x)=\bigl(\varphi_\mu(x)-\varphi_\mu(\tau)\bigr)^+,
\]
so that ${\rm supp}\,v_\mu \subset [-\tau,\tau]$. By definition of $\mathcal{L}_\mu$ and \eqref{lapos} we see that
\begin{align*}
\mathcal{E}(v_\mu) & = \frac12\int_{-\tau}^{\tau}|\varphi_\mu'|^2\dx - \frac1p\int_{-\tau}^{\tau}|\varphi_\mu-\varphi_\mu(\tau)|^p\dx \\
& = \int_{-\tau}^{\tau}\mathcal{L}_\mu\dx+\frac1p\int_{-\tau}^{\tau}(|\varphi_\mu|^p-|\varphi_\mu-\varphi_\mu(\tau)|^p)\dx \\
& \le \mathcal{E}(\varphi_\mu)+\frac1p\int_{-\tau}^{\tau}(|\varphi_\mu|^p-|\varphi_\mu-\varphi_\mu(\tau)|^p)\dx.
\end{align*}
Now, by the Mean Value Theorem,
\[
0 \le |\varphi_\mu(x)|^p - |\varphi_\mu(x)-\varphi_\mu(\tau)|^p \le p\varphi_\mu(\tau)|\varphi_\mu(x)|^{p-1} \qquad \forall x\in[-\tau,\tau],
\]
and thus, by \eqref{soliton},
\begin{align}
\label{stima}
\mathcal{E}(v_\mu) & \le \mathcal{E}(\varphi_\mu) + \varphi_\mu(\tau)\int_{-\tau}^{\tau}|\varphi_\mu(x)|^{p-1}\dx \nonumber \\
& = \mathcal{E}(\varphi_\mu) + \mu^{\alpha p-\beta} \varphi_1(\tau\mu^\beta)\int_{-\tau\mu^\beta}^{\tau\mu^\beta}|\varphi_1(y)|^{p-1}\,dy \nonumber \\
& \le \mathcal{E}(\varphi_\mu) + \mu^{2\beta+1}\varphi_1(\tau\mu^\beta)\int_{-\infty}^{+\infty}|\varphi_1(y)|^{p-1}\,dy\nonumber \\
& \le  \mathcal{E}(\varphi_\mu)  + C\mu^{2\beta+1}\varphi_1(\tau\mu^\beta),
\end{align}
having set $y = \mu^\beta x$ in the second line (we have also written $\alpha p -\beta$ as $2\beta+1$, which is the same). Now, by the form of $\varphi_1$ given in \eqref{form}, it is easy to check that
\begin{align*}
\mu^{2\beta+1} \varphi_1(\tau\mu^\beta) & = \frac{1}{\tau^{2+1/\beta}}(\tau \mu^\beta)^{2+1/\beta} \varphi_1(\tau\mu^\beta) \\
& \le \frac{C}{\tau^{2+1/\beta}}e^{-C\tau\mu^\beta} = \frac{C}{\ell^{2+1/\beta}}e^{-C\ell\mu^\beta},
\end{align*}
where all the constants $C$ depend only on $p$. Setting 
\begin{equation}
\label{erreelle}
r_\ell(\mu) := \frac{C}{\ell^{2+1/\beta}}e^{-C\ell\mu^\beta},
\end{equation}
we can write \eqref{stima} as
\[
\mathcal{E}(v_\mu) \le \mathcal{E}(\varphi_\mu)+ r_\ell(\mu),
\]
which then holds for every $\mu \ge\mu_1$. The quantity on the right--hand side is negative as soon as $\mu_1$ is large enough, since 
as $\mu\to \infty$, $r_\ell(\mu) \to 0$  and $\mathcal{E}(\varphi_\mu) \to -\infty$ by \eqref{levsol} . Finally, setting
\[
\psi(x):=\frac{\sqrt\mu}{\|v_\mu\|_{L^2(\R)}}v_\mu(x-\tau),
\]
we find that $\psi \in H^1(\R)$,  $\|\psi\|_{L^2(\R)}^2 =\mu$, ${\rm supp}\,\psi \subset [0,2\tau]= [0,\ell]$ and
\[
\mathcal{E}(\psi) \le \mathcal{E}(v_\mu) \le  \mathcal{E}(\varphi_\mu)+ r_\ell(\mu),
\]
because $\sqrt\mu >\|v_\mu\|_{L^2(\R)}$ and $\mathcal{E}(v_\mu)$ is negative. This completes the proof.
\end{proof}

\begin{remark}
\label{decrease}
For future reference we note that the quantity $r_\ell(\mu)$ defined in \eqref{erreelle} is decreasing both in $\mu$ and in $\ell$.
\end{remark}

\section{The Palais--smale condition}\label{palais}

In this section we analyze the Palais--Smale condition for the functional $E_M$. To be precise, we recall the following definition.

\begin{definition}
\label{ps} Let $c\in \R$. We say that a sequence $(u_n)\subset M$ is a \emph{Palais--Smale sequence} for $E_M$ at level $c$ if, as $n\to \infty$,
\begin{itemize}
\item[(i)] $E_M(u_n) \to c$,
\item[(ii)] $\|E_M'(u_n)\|_{T'_{u_n}M}\to 0$.
\end{itemize}
We say that $E_M$ satisfies the \emph{Palais--Smale condition} at level $c$ (shortly, $E_M$ satisfies (PS)$_c$) if every Palais--Smale
sequence at level $c$ admits a subsequence converging in $M$.
\end{definition}

The analysis of the Palais--Smale condition takes advantage of the following Gagliardo--Nirenberg inequality on graphs (for a proof see \cite{ast2,T}).

\begin{proposition}
For every $p\in[2,\infty)$ there exists a constant $\mathcal{C}_p>0$ (depending only on $p$) such that
\begin{equation}
\label{GN}
\|u\|_{L^p(\G)}^p \leq \mathcal{C}_p\|u\|_{L^2(\G)}^{\frac{p}{2}+1}\|u'\|_{L^2(\G)}^{\frac{p}{2}-1}\qquad\forall u\in H
\end{equation}
\end{proposition}

This inequality implies that every Palais--Smale sequence is bounded in $H$. Indeed (and more generally), from \eqref{GN},
\[
E_M(u) \ge \frac12\|u'\|_{L^2(\G)}^2  -\frac1p \mathcal{C}_p\mu^{\frac{p+2}{4}}\|u'\|_{L^2(\G)}^{\frac{p}{2}-1} \qquad\forall u\in M.
\]
Since $p\in(2,6)$,
\begin{equation}
\label{sublevel}
\|u\|^2 \leq C(1+E_M(u)) \qquad\forall u\in M,
\end{equation}
(for some $C$ depending on $p$ and $\mu$), showing that the sublevel sets of $E_M$ are bounded.

\begin{remark}
Since every Palais--Smale sequence $u_n$ is bounded, in view of Remark \ref{bdps},  the second condition in  Definition \ref{ps} can be written
more conveniently as $J(u_n) \to 0$ in $H'$. We will do so in the next proposition.
\end{remark}

\begin{proposition}
For every $\mu>0$, the functional $E_M$ satisfies (PS)$_c$ if and only if $c<0$.
\end{proposition}

\begin{proof}
We first show that the Palais--Smale condition does not hold at nonnegative levels. For every $c,\mu>0$, let $a>0$ and $c_n>0$ ($n\in \N$) be defined by
\[
a^2 = 2c/\mu, \qquad c_n^2 = \sqrt{2c\mu}/n\pi.
\]
Identify as usual one of the  half--lines of $\G$ with the interval $[0,+\infty)$,  and define the sequence
\[
u_n(x) =
\begin{cases}
c_n \sin (ax) & \text{if } x \in [0, 2n\pi/a] \\
0 & \text{otherwise on } \G.
\end{cases}
\]
Elementary computations show that $u_n$ is a Palais--Smale sequence for $E_M$ at level $c>0$ that has no converging subsequences in $M$, since $u_n\to 0$ uniformly on $\G$. Similarly, if $\xi\in C_0^1(\R^+)$ with $\|\xi\|_{L^2(\R^+)}^2=\mu>0$ and 
\[
u_n(x) =
\begin{cases}
n^{-1/2}\xi(x/n) & \text{if } x\in [0,+\infty)\\
0 & \text{otherwise on } \G,
\end{cases}
\]
it easy to see that $u_n$ is a Palais--Smale sequence for $E_M$ at level $0$ that again has no converging subsequences in $M$.

The situation changes at negative levels. Let $(u_n) \subset M$ be a Palais--Smale sequence at level $c<0$. Since it is bounded in $H$, up to subsequences, $u_n \rightharpoonup u$ in $H$ and $u_n\rightarrow u$ in $L_{loc}^q(\G)$, for all $q \in [1,\infty]$. As a consequence, $u_n \rightarrow u$ in $L^p(\K)$ and, by weak lower semicontinuity,
\[
 E(u)\leq\liminf_n E_M(u_n)=c<0.
\]
Since $E(0)=0$, this shows that $u\not\equiv0$. The boundedness of $u_n$ also implies, via \eqref{lam}, that $\lambda_n$ is bounded; therefore, up to subsequences, $\lambda_n \to \lambda$ as $n\to \infty$. By \eqref{juu}, $J(u_n)u_n =0$ for every $n$, so that
\begin{align*}
c+o(1) & = E_M(u_n)-\frac1pJ(u_n) u_n=\left(\frac12-\frac1p\right)\int_{\G}|u_n'|^2\dx-\frac{\mu}{p}\lambda_n \\
& \ge -\frac{\mu}{p}\lambda_n =-\frac{\mu}{p}\lambda + o(1).
\end{align*}
Hence, $\lambda >0$. Let now $A(u):H\to \R$ be defined by 
\[
A(u)v = \int_\G u'v'\dx + \lambda\int_\G u v\dx.
\]
By weak convergence of $u_n$ we have $A(u)(u_n-u) \to 0$ as $n\to\infty$. Therefore
\begin{align*}
o(1) & = \left(J(u_n)-A(u)\right)(u_n-u)\\ 
& = \int_\G |u_n' -u'|^2\dx-\int_\K |u_n|^{p-2}\,u_n (u_n-u)\dx +\lambda_n\int_\G u_n(u_n-u)\dx\\
& \phantom{aa} -\lambda \int_\G u(u_n-u)\dx \\
& = \int_\G |u_n' -u'|^2\dx +\lambda \int_\G |u_n-u|^2\dx +o(1),
\end{align*}
by strong convergence of $u_n$ to $u$ in $L^p(\K)$ and convergence of $\lambda_n$ to $\lambda$.
Since $\lambda$ is positive, this shows that $u_n \to u$ in $M$. 
\end{proof}

\section{Proof of the main result}\label{proofs}

We are going to prove Theorem \ref{main} by means of a classical result in Critical Point Theory (e.g. Proposition 10.8 in  \cite{Ambrosetti}).  First, we recall the notion of {\em genus} due to M. A. Krasnosel'skii (for a complete discussion, see \cite{Krasnoselski,Rabinowitz}).

\begin{definition}
Let $\mathcal{A}$ be the family of sets $A\subset H\backslash\{0\}$ such that $A$ is closed and symmetric (namely, $u\in A \Rightarrow -u\in A$). For every $A\in\mathcal{A}$, the {\em genus} of $A$ is the natural number defined by
\begin{equation}
\label{genus}
\gamma(A):=\min\{n\in\N:\exists\varphi:A\rightarrow\R^n\backslash\{0\},\, \varphi\mbox{ continuous and odd}\}.
\end{equation}
If no $\varphi$ as in \eqref{genus} exists, one sets $\gamma(A) = +\infty$.
\end{definition}

Given $\mu>0$, let $M$ be as in \eqref{manif}. Define, for $j\in \N$,
\[
\Gamma_j := \{A \subset M : A\in {\mathcal A},\; A \text{ is compact and } \gamma(A) \ge j\}
\] 
and
\[
c_j := \inf_{A\in \Gamma_j} \max_{u\in A} E_M(u).
\]

Note that $-\infty < c_1 \le c_2\le \dots$, the first inequality following from the fact that $E_M$ is bounded below by \eqref{sublevel} and the others by definition of the $\Gamma_j$'s. By the quoted result in \cite{Ambrosetti}, each $c_j$ is a critical level of $E_M$ provided that the Palais--Smale condition holds at $c_j$. Also, if $c_j=c_{j+1}=\dots = c_{j+q}$ for some $q\ge 1$, then the set $K=\{u \in M : E_M(u) = c_j,\; E'_M(u)=0\}$ satisfies $\gamma(K) \ge q+1$ and, in particular,  is infinite.

Now, since the functional $E_M$ satisfies the (PS)$_c$ conditions only if $c<0$, in order to conclude we have to show that some of the min--max classes $\Gamma_j$ work at negative levels, and this is where the choice of $\mu$ will become essential. 

\medskip
We start with a simple lemma, whose proof is omitted.

\begin{lemma}
\label{sk}
For every $k\in \N$, let $S^{k-1} = \{\theta \in \R^k : |\theta| =1\}$. Then,
\[
\min_{\theta\in S^{k-1}} \sum_{j=1}^k|\theta_j|^p = k^{1-p/2} \qquad \text{for every }\quad p\ge 2.
\]
\end{lemma}

We are now ready to carry out the proof of the main result.

\begin{proof}[Proof of Theorem~\ref{main}]
Let $e$ be any edge of $\K$, identified with the interval $[0,L]$. For $k\in\N$, let $\mu_{L/k}$ be the number provided by Proposition \ref{approx} (with the choice $\ell=L/k$), and define
\begin{equation}
\label{muk}
\mu_k = k\mu_{L/k}.
\end{equation}
Now, let $M$ be the manifold \eqref{manif} corresponding to any choice of $\mu\ge \mu_k$ and, accordingly,
let $E_M$ be the functional $E$ restricted to $M$. Since $\mu/k\ge \mu_k/k = \mu_{L/k}$, Proposition \ref{approx} yields a function $\psi\in H^1(\R)$ such that
\begin{itemize}
\item[(i)] $\|\psi\|_{L^2(\R)}^2 = \mu/k$,
\item[(ii)] ${\rm supp }\,\psi \subset [0, L/k]$,
\item[(iii)] $\mathcal{E}(\psi) \le \mathcal{E}(\varphi_{\mu/k}) + r_{L/k}(\mu/k)<0$.
\end{itemize}
We now construct, with the help of $\psi$, a compact symmetric set $A\subset M$ such that $\gamma(A)\ge k$ and $\max_{u\in A} E_M(u) <0$ (this will show that for all $j\le k$, $\Gamma_j \ne \emptyset$ and  $c_j<0$). For every $j =1,\dots, k$, define $\psi_j : \G \to \R$ as
\[
\psi_j(x) =
\begin{cases} \psi(x-(j-1)L/k) & \text {if } x\in [(j-1)L/k,jL/k] \\
0 & \text{otherwise on } \G.
\end{cases}
\]
Each function $\psi_j$ is therefore supported in the subset of the edge $e$ identified with the subinterval $[(j-1)L/k,jL/k]$. Note that $\psi_i(x)\psi_j(x) = 0$ for every $x\in \G$ and every $i\ne j$. Clearly, $\psi_j \in H$ and 
\[
\|\psi_j\|_{L^2(\G)}^2 =  \|\psi\|_{L^2(\R)}^2 = \mu/k, \qquad  j=1,\dots,k.
\]
Consider the map $h : S^{k-1}\rightarrow H$ defined by
\begin{equation}
\label{acca}
h(\theta)(x) = \sqrt{k} \sum_{j=1}^k \theta_j\psi_j(x).
\end{equation}
Obviously, $h$ is continuous, odd, $0\not\in h(S^{k-1})$, and $h(\theta)$ is supported in $[0,L]$. Moreover, for every $\theta \in S^{k-1}$,
\begin{align*}
\int_\G |h(\theta)(x)|^2\dx &= \int_0^L  |h(\theta)(x)|^2\dx = k \sum_{j=1}^k \theta_j^2 \int_{(j-1)L/k}^{jL/k}|\psi_j(x)|^2\dx \\
& =  k \sum_{j=1}^k \theta_j^2 \int_{0}^{L/k}|\psi(x)|^2\dx = k\|\psi\|_{L^2(\R)}^2\sum_{j=1}^k \theta_j^2 = \mu
\end{align*}
by (i). Therefore $h$ maps $S^{k-1}$ into $M$. Similarly,
\[
\int_\G |h(\theta)'(x)|^2\dx = k\int_{0}^{L/k}|\psi'(x)|^2\dx \sum_{j=1}^k \theta_j^2 = k\int_\R |\psi'(x)|^2\dx
\]
and
\begin{align*}
\int_\K |h(\theta)(x)|^p\dx   =  k^{p/2} \int_\R |\psi(x)|^p\dx\sum_{j=1}^k |\theta_j|^p \ge k \int_\R |\psi(x)|^p\dx
\end{align*}
by Lemma \ref{sk}. Therefore, for every $\theta\in S^{k-1}$,
\begin{align}
\label{neglev}
E_M(h(\theta)) &= \frac12 \int_\G |h(\theta)'(x)|^2\dx - \frac1p \int_\K |h(\theta)(x)|^p\dx \nonumber \\
& \le \frac{k}{2} \int_\R |\psi'(x)|^2\dx   -\frac{k}{p}  \int_\R |\psi(x)|^p\dx \nonumber \\ 
& = k\mathcal{E}(\psi) \le  k(\mathcal{E}(\varphi_{\mu/k})+r_{L/k}(\mu/k))<0,
\end{align}
where the last inequalities follow from (iii). Set now $A = h(S^{k-1})$. Then $A$ is compact, symmetric and $A\subset M$. By well--known properties of the genus (see e.g. Chapter 7 in \cite{Rabinowitz})
\[  
\gamma(A) \ge \gamma(S^{k-1})=k,
\]
so that $A\in \Gamma_k$. By \eqref{neglev}, $\max_{u\in A} E_M(u) <0$, showing that
\[
c_k = \inf_{A\in \Gamma_k} \max_{u\in A} E_M(u) <0.
\]
Thus $E_M$ satisfies the Palais--Smale conditions at all levels $c_1 \le\dots\le c_k$. We then apply Proposition 10.8 in \cite{Ambrosetti} and find,
for every $j=1,\dots,k$, at least one pair of critical points $\pm u_j$ for $E_M$ at level $c_j$. If two or more of the $c_j$'s coincide, by the quoted result we deduce that $E_M$ has infinitely many critical points.  

Next we show that the estimates \eqref{bslevel} are satisfied. The computations carried out so far for $c_k$ can be repeated, exactly in the same way, for each $c_j$ with $j< k$. Hence, working  with  any $S^{j-1}$, 
\eqref{neglev} reads
\[
E_M(h(\theta)) \le  j(\mathcal{E}(\varphi_{\mu/j})+r_{L/j}(\mu/j)), \qquad j=1,\,\dots, k,
\]
where $h:S^{j-1}\to H$ is defined according to \eqref{acca}. By \eqref{levsol},
\[
\mathcal{E}(\varphi_{\mu/j}) = \frac1{j^{2\beta+1}}\mathcal{E}(\varphi_\mu),
\]
so that, for every $\theta\in S^{j-1}$,
\[
E_M(h(\theta)) \le  \frac1{j^{2\beta}}\mathcal{E}(\varphi_\mu) + jr_{L/j}(\mu/j)).
\]
By Remark \ref{decrease}
\[
jr_{L/j}(\mu/j)) \le kr_{L/k}(\mu/k)):= \sigma_k(\mu) \qquad  j=1,\,\dots, k,
\]
hence we obtain
\[
c_j \le  \frac1{j^{2\beta}}\mathcal{E}(\varphi_\mu)  + \sigma_k(\mu), \qquad  j=1,\,\dots, k,
\]
and also the required level estimate is proved.

Finally, the sign of the Lagrange multiplier $\lambda_j$ associated to $u_j$ can be immediately deduced from the fact that the energy of $u_j$ is negative via the definition of $\lambda_j$:
\[
\lambda\mu = \int_\K |u_j|^p\dx -\int_\G |u_j'|^2\dx > \frac2p\int_\K |u_j|^p\dx  -\int_\G |u_j'|^2\dx = -2E_M(u_j)>0.
\]
\end{proof}

A few comments are in order. First, we note that the mass threshold $\mu_k$ defined in \eqref{muk} is a decreasing function of the length $L$
of the edge chosen for the construction carried out in the proof. Therefore, in order to minimize $\mu_k$, and obtain the result
for masses as small as possible, it is convenient to choose the longest edge in the compact core of the graph. Alternatively, with a minor modification of the argument,
one could also ``place'' the $k$ copies of the cut--off soliton used for the construction of the set $A$ on {\em different} edges of the compact core, when this is preferable in order to minimize $\mu_k$.

These considerations also show that there is no need to require that the nonlinearity be present on  the whole  compact core. A single edge is sufficient in order to obtain the same result, as anticipated in Remark \ref{single}.


\begin{thebibliography}{99}

\bibitem{acfn}
Adami R., Cacciapuoti C., Finco D., Noja D.,
Constrained energy minimization and orbital stability for the NLS equation on a star graph.
\emph{Ann. Inst. H. Poincar\'e Anal. Non Lin\'eaire} {\bf 31} (2014), no. 6, 1289--1310.

\bibitem{adft}
Adami R., Dell'Antonio G., Figari R., Teta A.,
The Cauchy problem for the Schr\"odinger equation in dimension three with concentrated nonlinearity.
\emph{Ann. Inst. H. Poincar\'e Anal. Non Lin\'eaire} {\bf 20} (2003), no. 3, 477--500.

\bibitem{ast1}
Adami R., Serra E., Tilli P., 
NLS ground states on graphs.
\emph{Calc. Var. Partial Differential Equations} {\bf 54} (2015) no. 1, 743--761.

\bibitem{ast2}
Adami R., Serra E., Tilli P., 
Threshold phenomena and existence results for NLS ground states on graphs.
\emph{ArXiv: 1505.03714} (2015).

\bibitem{at}
Adami R., Teta A.,
A class of nonlinear Schr\"odinger equations with concentrated nonlinearity.
\emph{J. Funct. Anal.} {\bf 180} (2001), no. 1, 148--175.

\bibitem{Ambrosetti}
Ambrosetti A., Malchiodi A.,
\emph{Nonlinear analysis and semilinear elliptic problems}.
Cambridge Studies in Advanced Mathematics, 104. Cambridge University Press, Cambridge, 2007.

\bibitem{bariola}
Bartsch T., de Valeriola S.,
Normalized solutions of nonlinear Schr\"odinger equations.
\emph{Arch. Math. (Basel)} {\bf 100} (2013), no. 1, 75--83.

\bibitem{bl2}
Berestycki H., Lions P.--L.,
Nonlinear scalar field equations. II. Existence of infinitely many solutions.
\emph{Arch. Rational Mech. Anal.} {\bf 82} (1983), no. 4, 347--375.

\bibitem{berkolaiko}
Berkolaiko G., Kuchment P.,
\emph{Introduction to quantum graphs}.
Mathematical Surveys and Monographs, 186. AMS, Providence, RI, 2013.

\bibitem{CFN}
Cacciapuoti C., Finco D., Noja D.,
Topology--induced bifurcations for the nonlinear Schr\"odinger equation
on the tadpole graph.
\emph{Phys. Rev. E} {\bf 91} (2015), no. 1, article number 013206, 8pp.

\bibitem{Cacciapuoti}
Cacciapuoti C., Finco D., Noja D., Teta A.,
The NLS equation in dimension one with spacially concentrated nonlinearities: the pointlike limit.
\emph{Lett. Math. Phys.} {\bf 104} (2014), no. 12, 1557--1570.

\bibitem{cazenave}
Cazenave T.,
{\em Semilinear Schr\"odinger equations}.
Courant Lecture Notes in Mathematics, 10. AMS, Providence, RI, 2003.

\bibitem{caz-lio}
Cazenave T., Lions P.--L.,
Orbital stability of standing waves for some nonlinear Schr\"odinger equations.
\emph{Comm. Math. Phys.} {\bf 85} (1982), no. 4, 549--561.

\bibitem{Deimling}
Deimling K.,
\emph{Nonlinear functional analysis}.
Springer--Verlag, Berlin, 1985.

\bibitem{gnusmi}
Gnutzmann S., Smilansky U.,
Quantum graphs: Applications to quantum chaos and universal spectral statistics.
\emph{Adv. Phys.} {\bf 55} (2006), no. 5--6, 527--625.

\bibitem{der}
Gnutzmann S., Smilansky U., Derevyanko S.,
Stationary scattering from a nonlinear network.
\emph{Phys. Rev. A} {\bf 83} (2011), no. 3, article number 033831, 6pp.

\bibitem{josinio}
Jona--Lasinio G., Presilla C., Sj\"ostrand J.,
On Schr\"odinger equations with concentrated nonlinearities.
\emph{Ann. Physics} {\bf 240} (1995), no. 1, 1--21.

\bibitem{kottos}
Kottos T., Smilansky U.,
Periodic orbit theory and spectral statistics for quantum graphs.
\emph{Ann. Physics} {\bf 274} (1999), no. 1, 76--124.

\bibitem{Krasnoselski}
Krasnosel'skii M. A.,
\emph{Topological methods in the theory of nonlinear integral equations}.
The Macmillan Co., New York, 1964.

\bibitem{Kuchment}
Kuchment P.,
Quantum graphs. I. Some basic structures.
\emph{Waves Random Media} {\bf 14} (2004), no. 1, 107--128.

\bibitem{nier}
Nier F.,
The dynamics of some quantum open systems with short--range nonlinearities.
\emph{Nonlinearity} {\bf 11} (1998), no. 4, 1127--1172.

\bibitem{noja}
Noja D.,
Nonlinear Schr\"odinger equation on graphs: recent results and open problems.
\emph{Philos. Trans. R. Soc. Lond. Ser. A Math. Phys. Eng. Sci.} {\bf 372} (2014), no. 2007, 20130002, 20pp.

\bibitem{Rabinowitz}
Rabinowitz P.H.,
\emph{Minimax methods in critical point theory with applications to differential equations}.
CBMS Regional Conference Series in Mathematics, 65. AMS, Providence, RI, 1986.

\bibitem{suldue}
Sulem C., Sulem P.--L.,
\emph{The nonlinear Schr\"odinger equation. Self--focusing and wave collapse}.
Applied Mathematical Sciences, 139. Springer--Verlag, New York, 1999.

\bibitem{T}
Tentarelli L.,
NLS ground states on metric graphs with localized nonlinearities.
\emph{J. Math. Anal. Appl.} {\bf 433} (2016), no. 1, 291--304.

\bibitem{zakharov}
Zakharov V.E., Shabat A.B.,
Exact theory of two--dimensional self focusing and one--dimensional self--modulation of waves in nonlinear media.
\emph{Soviet Physics JETP} {\bf 34} (1972), no. 1, 62--69.


\end{thebibliography}
\end{document}